\documentclass{siamltex}

\usepackage{amsmath}
\usepackage{amssymb,mathtools}
\usepackage{graphicx}
\usepackage{epstopdf}
\usepackage{color}
\usepackage{graphicx}
\usepackage{lipsum}
\usepackage{subcaption}
\usepackage{hyperref}
\usepackage[compatible]{algpseudocode}
\usepackage{algorithmicx}
\usepackage[T1]{fontenc}
\usepackage[dvipsnames]{xcolor}
\usepackage{listings}
\usepackage{multirow}
\usepackage{pifont}
\usepackage{scalerel}
\usepackage{soul,color}
\usepackage{comment}
\usepackage{nicematrix}
\usepackage{tikz}
\usepackage{pgfplots}
\pgfplotsset{compat=1.18}
\usetikzlibrary{positioning}

\newcommand{\Atld}{\widetilde{A}}

\newcommand{\btld}{\widetilde{b}}

\newcommand{\dhat}{\widehat{d}}
\newcommand{\ddelta}{\delta d}
\newcommand{\ddeltah}{\delta \widehat{d}}

\newcommand{\Rhat}{\widehat{R}}

\newcommand{\rhat}{\widehat{r}}
\newcommand{\rdelta}{\delta r}

\newcommand{\shat}{\widehat{s}}
\newcommand{\xhat}{\widehat{x}}
\newcommand{\xdelta}{\delta x}

\newcommand{\Yhat}{\widehat{Y}}

\newcommand{\Zhat}{\widehat{Z}}

\newcommand{\what}{\widehat{w}}

\newtheorem{cor}[theorem]{Corollary}

\newlength\figureheight
\newlength\figurewidth

\usepackage[section]{algorithm}
\renewcommand{\algorithmicrequire}{\textbf{Input: }}
\renewcommand{\algorithmicensure}{\textbf{Output: }}

\renewcommand{\algorithmiccomment}[1]{\bgroup\hfill//~#1\egroup}

\title{Mixed precision sketching for least-squares problems and its application in GMRES-based iterative refinement}
\author{Erin Carson\thanks{Faculty of Mathematics and Physics, Charles University, carson@karlin.mff.cuni.cz.}  \and  Ieva Dau\v{z}ickait\.{e}\thanks{CERFACS, Toulouse, France, dauzickaite@cerfacs.fr.}\\
Both authors were supported by Charles University Research Centre program No. UNCE/24/SCI/005, the Exascale Computing Project (17-SC-20-SC), a collaborative effort of the U.S. Department of Energy Office of Science and the National Nuclear Security Administration, and by the European Union (ERC, inEXASCALE, 101075632). Views and opinions expressed are those of the authors only and do not necessarily reflect those of the European Union
or the European Research Council. Neither the European Union nor the granting authority can be held responsible
for them.}

\begin{document}

\maketitle

\begin{abstract}
   Sketching-based preconditioners have been shown to accelerate the solution of dense least-squares problems with coefficient matrices having substantially more rows than columns. The cost of generating these preconditioners can be reduced by employing low precision floating-point formats for all or part of the computations. We perform finite precision analysis of a mixed precision algorithm that computes the $R$-factor of a QR factorization of the sketched coefficient matrix. Two precisions can be chosen and the analysis allows understanding how to set these precisions to exploit the potential benefits of low precision formats and still guarantee an effective preconditioner. If the nature of the least-squares problem requires a solution with a small forward error, then mixed precision iterative refinement (IR) may be needed. For ill-conditioned problems the GMRES-based IR approach can be used, but good preconditioner is crucial to ensure convergence. We theoretically show when the sketching-based preconditioner can guarantee that the GMRES-based IR reduces the relative forward error of the least-squares solution and the residual to the level of the working precision unit roundoff. Small numerical examples illustrate the analysis.
\end{abstract}



\section{Introduction}

Let $A$ be an $m\times n$ matrix with full-rank and $m \gg n$, let $b$ be a length-$m$ vector, and suppose we want to solve the least squares problem 
\begin{equation}\label{eq:LS_problem}
    \min_x \Vert b-Ax\Vert_2.
\end{equation}
A variety of methods, such as Least Squares QR (LSQR) can be used to find $x$ that minimizes the residual \cite{paige1982lsqr}. Rokhlin and Tygert showed that preconditioning a dense $A$ with the $R$ factor of the QR decomposition of a sketched $A$ can greatly reduce the condition number \cite{rokhlin2008tygert}. An efficient implementation of these ideas in a BLENDENPIK solver has been shown to reduce the LSQR iteration count and the wall-clock time \cite{avron2010blendenpik}. The authors in \cite{georgiou2023mixed} employ mixed precision to generate this preconditioner and thus improve performance. 

The available finite precision analysis for generating the randomized preconditioner assumes that all operations are performed in a uniform precision, and $A$ is well-conditioned with respect to this precision. In this work, we provide a more general analysis of a mixed precision setting where the sketching operation and the QR decomposition are computed in two possibly different precisions. Our main result assumes that $A$ is not too ill-conditioned with respect to the sketching precision, but we provide some comments on cases that do not satisfy this assumption.

In recent work, randomized preconditioning has been combined with fixed precision iterative refinement (IR) to ensure a backward stable solution to \eqref{eq:LS_problem} \cite{epperly2024fast}. We consider combining the mixed precision randomized preconditioning with IR with a different objective: we aim to refine both the computed solution and residual of \eqref{eq:LS_problem} so that their relative \textit{forward error} reaches the level of unit roundoff of the working precision. This may be particularly relevant when a low precision format is used as the working precision. A popular way to achieve this goal is through the use of a mixed precision GMRES-based iterative refinement scheme for least-squares (LSIR). Here, we first solve the least-squares problem via LSQR with randomized preconditioning and then use iterative refinement on the augmented system
\[
\begin{bmatrix}
 I & A\\ A^T & 0
\end{bmatrix}
\begin{bmatrix}
r\\  x
\end{bmatrix}
=
\begin{bmatrix}
 b \\0
\end{bmatrix}.
\]
The augmented system LSIR approach was introduced and analysed (when solved via a QR decomposition of $A$) by Bj\"{o}rck 
\cite{bjorck1967solving} and is the only approach where the residual is refined explicitly. The scheme was extended to a more general setting
in previous work \cite{carson2020three}, where a linear system with the augmented coefficient matrix in each refinement step is solved via left-preconditioned GMRES. The left preconditioner was constructed of QR factors of $A$, assumed to be computed via Householder QR in some precision $u_f$. Despite that the QR factorization is computed in a potentially lower precision, its cost may still be significant (i.e., reducing the precision does not reduce the latency cost). 
Here we seek to use the already computed randomized $R$ factor in the preconditioner. 
We define the preconditioned augmented system to be 
\begin{equation}\label{eq:ir_system_split_prec}
\underbrace{\begin{bmatrix}
I & 0 \\ 0 & R^{-T} 
\end{bmatrix}}_{M_L^{-1}}
\underbrace{\begin{bmatrix}
 I & A\\ A^T & 0
\end{bmatrix}}_{\Atld}
\underbrace{\begin{bmatrix}
I & 0 \\ 0 & R^{-1} 
\end{bmatrix}}_{M_R^{-1}}
\begin{bmatrix}
\hat{r}\\  \hat{x}
\end{bmatrix}
=
\begin{bmatrix}
I & 0 \\ 0 & R^{-T} 
\end{bmatrix}
\begin{bmatrix}
 b \\0
\end{bmatrix},
\end{equation}
with 
\[
\begin{bmatrix}
I & 0 \\ 0 & R^{-1} 
\end{bmatrix}
\begin{bmatrix}
\hat{r}\\  \hat{x}
\end{bmatrix}
= 
\begin{bmatrix}
r\\ x
\end{bmatrix}.
\]

IR convergence can be guaranteed if the augmented systems are solved via a backward stable solver and the relative forward error is reduced in each iteration \cite{carson2020three}. 
GMRES with right preconditioning may be unstable \cite{arioli2007note} and we thus use its flexible variant FGMRES to solve \eqref{eq:ir_system_split_prec}. We note that other approaches to mixed precision least-squares iterative refinement have been proposed, however only the augmented system approach has convergence guarantees for both $x$ and $r$; see, e.g., \cite{bjorck1967solving,carson2024comparison}.

The goal of this work is twofold: to analyze the randomized sketching in two precisions and to determine when this approach to LSIR is guaranteed to converge, and give performance benefits. The theoretical analysis uses results for FGMRES backward stability and requires bounds for the condition number of the preconditioned augmented matrix.
Our primary contributions are as follows. 
\begin{itemize}
    \item We extend the finite precision analysis for generating a randomized preconditioner to the two precision case.
    \item  We provide the first theoretical proof that something other than a regular QR decomposition can provide a preconditioner of sufficient quality for LSIR to converge. 
    \item  We provide the first theoretical proof that FGMRES with split preconditioning works for LSIR rather than GMRES restricted to left preconditioning.
\end{itemize}

The paper is structured as follows. We review theoretical results on sketching in Section~\ref{sec:sketching_background} and analyze the two precision version of the sketched QR decomposition algorithm in Section~\ref{sec:sketching_finite_precision}. The LSIR method with a  randomized preconditioner is analyzed in Section~\ref{sec:lsir}, followed by numerical examples in Section~\ref{sec:numerics}. We conclude in Section~\ref{sec:conclusions}.

\section{Background on sketching}\label{sec:sketching_background}

The sketching matrix is often chosen so that the sketch preserves some qualities of $A$. Preserving the norm of a matrix-vector product is important in many applications. This may be achieved using a subspace embedding, which we define in the following.
\begin{definition}
    Consider a subspace $\mathbb{F} \subseteq \mathbb{R}^m$, a linear operator $\Omega : \mathbb{R}^{m} \to \mathbb{R}^{s}$, and a distortion parameter $\epsilon \in (0,1)$. $\Omega$ is called an $l_2$ $\epsilon$-subspace embedding if for every $x \in \mathbb{F}$ we have
    \begin{equation*}
        \sqrt{1 - \epsilon} \, \Vert x \Vert_2 \leq \Vert \Omega x \Vert_2 \leq \sqrt{1 + \epsilon} \, \Vert x \Vert_2.
    \end{equation*}
\end{definition}
Note that constructing such an $\Omega$ requires some knowledge of $\mathbb{F}$. To relax this requirement, the following subset of subspace embeddings is often considered.
\begin{definition}
     A linear operator $\Omega : \mathbb{R}^{m} \to \mathbb{R}^{s}$ is called an oblivious $(\epsilon,\delta,n)$-subspace embedding if it is a subspace embedding for any $n$-dimensional $\mathbb{F} \subseteq \mathbb{R}^m$ with probability at least $1-\delta$.
\end{definition}

We can set $\mathbb{F} = range(A)$ and then for every $y \in \mathbb{R}^n$ we have
    \begin{equation*}
        \sqrt{1 - \epsilon} \, \Vert Ay \Vert_2 \leq \Vert \Omega Ay \Vert_2 \leq \sqrt{1 + \epsilon} \, \Vert Ay \Vert_2.
    \end{equation*}
Using this with the definition of largest and smallest singular values, we can easily obtain the bounds
\begin{gather*}
    \sqrt{1 - \epsilon} \, \sigma_{min} (A) \leq \sigma_{min} (\Omega A) \leq \sigma_{max} (\Omega A) \leq  \sqrt{1 + \epsilon} \,  \sigma_{max} (A), \\
    \frac{\sigma_{min} (\Omega A)}{\sqrt{1 + \epsilon} } \leq \sigma_{min} ( A) \leq \sigma_{max} ( A) \leq  \frac{\sigma_{max} (\Omega A)}{\sqrt{1 - \epsilon} }, 
\end{gather*}
which give
\begin{equation*}
     \sqrt{ \frac{1 - \epsilon}{1 + \epsilon} } \kappa(A) \leq \kappa(\Omega A) \leq   \sqrt{ \frac{1 + \epsilon}{1 - \epsilon} }  \kappa(A).
\end{equation*}
Thus sketching with an $\epsilon$-subspace embedding approximately preserves the extreme singular values and we can expect $\kappa(\Omega A)$ to be close to $\kappa(A)$.

An embedding that has been theoretically analyzed and widely used is a Gaussian matrix, namely, $\Omega \in \mathbb{R}^{s \times m}$ with independent entries drawn at random from $\mathcal{N}(0,1/s)$ with an appropriately set $s$. Note that $\Omega$ is dense and thus for large problems storage and computing matrix-vector products can be expensive. Good quality results for sketching tall and skinny dense matrices are usually obtained by setting $s=cn$ with $c$ being small; see, for example, \cite{georgiou2023mixed} where $c \leq 4$.

Another well-studied and frequently used embedding, which has a sparse structure, is a sparse sign matrix $\Omega = \sqrt{m/\zeta} \begin{pmatrix}
    \omega_1 & \omega_2 & \dots & \omega_m
\end{pmatrix}$, where $\omega_j$ has $\zeta$ (with $2\leq \zeta\leq s$) entries chosen randomly and uniformly and set to $1$ or $-1$ with probability $1/2$; all other entries in $\omega_j$ are set to zero, see, e.g. \cite[Section 9.2]{martinsson2020randomized},\cite{dong2023simpler},\cite{cohen2016nearly}. Such $\Omega$ thus randomly samples and adds/subtracts some rows of $A$. Although the theoretical results require more samples than for the Gaussian embeddings, reliable results are obtained with $s$ being a small multiple of $n$.

It is known that if $\Omega$ is an $s \times m$ matrix such that $\Omega A$ is full rank and $\Omega A = Q P$ is a decomposition where columns of the $s \times n$ matrix $Q$ are orthonormal and $P$ is any $n \times n$ matrix, then in infinite precision 
\[\kappa(AP^{-1}) = \kappa(\Omega U) = \kappa(\Omega Q_A), \]
where $A = U \Sigma V^T$ is the economic SVD and $A = Q_A R_A$ is the economic QR factorization; see \cite[Theorem 1]{rokhlin2008tygert} and \cite[Lemma 2.1]{meier2024sketch}. We can also show that
\begin{equation*}
    \resizebox{.99\hsize}{!}{$\Vert A P^{-1} \Vert_2 = \Vert R_A P^{-1} \Vert_2 =\Vert ( \Omega Q_A)^{\dagger} \Omega Q_A R_A P^{-1} \Vert_2 \\ 
    = \Vert ( \Omega Q_A)^{\dagger} Q P P^{-1} \Vert_2 = \Vert ( \Omega Q_A)^{\dagger} \Vert_2,$} 
\end{equation*}
where $(\Omega Q_A)^{\dagger}$ is the Moore-Penrose pseudoinverse of $\Omega Q_A$ and the last equality is due to the columns of $Q$ being the orthonormal basis for $\textrm{range}(\Omega Q_A)$. 
The reduction of the norm and the condition number is thus determined by how well the sketching operator approximates the basis for the range of $A$. In our case, we set $P$ to be the $R$-factor of the economic QR decomposition $\Omega A = QR$. The norm reduction can be expressed via the subspace embedding distortion parameter $\epsilon$ as \cite[Proposition 5.4]{kireeva2024randomized}
\begin{equation*}
    \Vert A R^{-1} \Vert_2 \leq \frac{1}{\sqrt{1 - \epsilon}} \quad \textrm{and} \quad
    \Vert (A R^{-1})^{\dagger} \Vert_2 \leq \sqrt{1 + \epsilon}.
\end{equation*}

\section{Sketching in two precisions}\label{sec:sketching_finite_precision}

We perform a finite precision analysis of Algorithm~\ref{alg:generating_R}, where the sketching and QR steps can use possibly different precisions with unit roundoffs $u_s$ and $u_{QR}$. We consider a general sketching operator $\Omega$, and assume that $A$ is full rank. $R$ denotes the $R$-factor of the economic QR decomposition of the sketched matrix $\Omega A$, that is, $\Omega A = QR$. A standard model of floating point arithmetic is used, where in the bounds we will make use of the quantities
\begin{gather*}
    \gamma_n^{(p)} = \frac{n u_p}{1 - n u_p} \quad  \textrm{and} \quad \widetilde{ \gamma}_n^{(p)} = \frac{cn u_p}{1 - c n u_p},
\end{gather*}
where $c$ is a small constant that does not depend on $n$ \cite[Section 2.2]{high:ASNA2}. We make the standard assumption that no overflow or underflow occurs. In the following, hats denote computed quantities, that is, $\Rhat$ is the computed version of $R$.

\begin{algorithm}
\caption{Randomized sketching based approximation of the $R$ factor of the QR factorization of matrix $A$ in precisions $u_s$ and $u_{QR}$ }\label{alg:generating_R}
\algorithmicrequire  $A \in \mathbb{R}^{m \times n}$ of full rank stored in precision $u_s$, sketching matrix/operator $\Omega$ such that $\Omega A \in \mathbb{R}^{s \times n}$ \\
\algorithmicensure  $R$ factor of the QR decomposition of sketched $A$
\begin{algorithmic}[1]
\State Compute the sketch $Y = \Omega A$ \Comment{$u_s$} \label{step:sketch}
\State Compute an economic Householder QR: $Y = QR$ \Comment{$u_{QR}$}
\end{algorithmic}
\end{algorithm}

\subsection{Computing $R$}
We obtain bounds in relation to the exact sketched matrix $\Omega A$ and the exact preconditioned matrix $A R^{-1}$. This allows us to use results that are available in the literature for the exact arithmetic case. We comment on the results in the following subsection.
\begin{theorem}\label{th:finite_recision_error_R}
    Consider $\Yhat = \Omega A + \Delta_s$ computed in step~\ref{step:sketch} of Algorithm~\ref{alg:generating_R}, where $\Delta_s$ accounts for the errors in casting $A$ to $u_s$ and computing the sketch. If $\Omega$, $u_s$, and $u_{QR}$ are set so that
    \begin{gather}
         (2 \log_2n + 4) n^{1/2} \widetilde{\gamma}_{sn}^{(QR)}  \kappa_2(\Yhat) < 1 \quad \text{ and } \label{eq:assumption_u_qr} \\
          (2 \log_2n + 4) n^{1/2} \Vert (\Omega A)^{\dagger} \Vert_2 \Vert \Delta_s \Vert_2 <1 \label{eq:assumption_u_s}
    \end{gather}
    and we denote
  \begin{align*}
        \beta = & \, \left( 1 + 18 (2 \log_2n + 4) n^{1/2} \widetilde{\gamma}_{sn}^{(QR)}  \kappa_2(\Omega A)\right)\\
     & \, \times \left(1 + 2(2 \log_2n + 4) n^{1/2} \Vert (\Omega A)^{\dagger} \Vert_2 \Vert \Delta_s \Vert_2 \right),
    \end{align*}
    then $\Rhat$ satisfies the following:
    \begin{align}
        \Vert \Rhat \Vert_2 & \leq \beta \Vert \Omega A \Vert_2, \label{eq:Rhat_norm_bound}\\
        \Vert \Rhat^{-1} \Vert_2 & \leq \beta \Vert (\Omega A)^{\dagger} \Vert_2, \label{eq:Rhat_inv_norm_bound}\\
        \kappa_2(\Rhat) & \leq \beta^2 \kappa_2(\Omega A), \label{eq:Rhat_cond_bound} \\
        \Vert A \Rhat^{-1} \Vert_2 & \leq \beta \Vert AR^{-1} \Vert_2, \label{eq:ARhat_inv_norm_bound}\\
        \Vert (A \Rhat^{-1})^{\dagger} \Vert_2 & \leq \beta \Vert (AR^{-1})^{\dagger} \Vert_2, \label{eq:ARhat_inv_pinv_norm_bound}\\
         \kappa_2(A \Rhat^{-1}) & \leq \beta^2 \kappa_2( A R^{-1}). \label{eq:ARhat_inv_cond_bound}
    \end{align}
\end{theorem}

\begin{proof}
    We perform the analysis in two steps. First, we express $\Rhat$ via the R factor of the exact economic QR decomposition  
\[\Yhat = Q_Y R_Y. \]
Then, we express $R_Y$ via the R factor of the exact sketched matrix \[ \Omega A = QR. \] 
We can do this by writing
\begin{gather}
     \Rhat = (I + \Gamma_1) R_Y  \text{ and } \label{eq:Rhat_R_Y} \\
     R_Y =  (I + \Gamma_2) R, \label{eq:R_Y_R}
\end{gather}
where $I$ is an $n \times n$ identity matrix, and $\Gamma_1$ and $\Gamma_2$ are upper triangular, and bounding $\Vert \Gamma_1 \Vert_2 $ and $\Vert \Gamma_2 \Vert_2$.

We start with $\Vert \Gamma_1 \Vert_2$, which requires considering the finite precision error in computing the QR decomposition. Standard results \cite[Theorem 19.4]{high:ASNA2} show that the Householder QR decomposition of $\Yhat$ returns $\Rhat$ such that 
\begin{gather*}
    \Yhat + \Delta_H = \Bar{Q} \Rhat, \textrm{ where} \\
    \| (\Delta_H)_j \|_2 \leq \widetilde{\gamma}_{sn}^{(QR)} \| (\Yhat)_j \|_2, \quad j=1:n,
\end{gather*}
and $\Bar{Q} \in \mathbb{R}^{s \times n}$ has orthonormal columns. Using $ \| B \|^2_F = \sum_j \| B_j \|^2_2$ we obtain
\begin{equation}\label{eq:delta_H_bound}
    \| \Delta_H \|_F \leq \widetilde{\gamma}_{sn}^{(QR)}  \|\Yhat \|_F.
\end{equation}
We proceed by considering the Cholesky decompositions of $\Yhat^T \Yhat$ and $(\Yhat + \Delta_H)^T (\Yhat + \Delta_H)$, namely,
\begin{equation*}
    \Yhat^T \Yhat = R_Y^T R_Y
\end{equation*}
and 
\begin{align*}
    (\Yhat + \Delta_H)^T (\Yhat + \Delta_H) = & \, \Yhat^T \Yhat + \underbrace{\Yhat^T \Delta_H + \Delta_H^T \Yhat + \Delta_H^T \Delta_H}_{E_c}\\
    = & \,  R_Y^T R_Y + E_c \\
    =& \Rhat^T \Rhat.
\end{align*}
Then using \eqref{eq:Rhat_R_Y} we have
\begin{equation*}
    R_Y^T R_Y + E_c = \Rhat^T \Rhat = R_Y^T(I + \Gamma_1)^T (I + \Gamma_1)R_Y
\end{equation*}
and by multiplying $R_Y^{-T}$ on the left and $R_Y^{-1}$ on the right we obtain
\begin{equation*}
    I + R_Y^{-T} E_c R_Y^{-1} = (I + \Gamma_1)^T (I + \Gamma_1).
\end{equation*}
In \cite[Theorem 3.1]{edelman1995parlett}, it is shown that 
\begin{equation*}
    \Vert \Gamma_1 \Vert_2 \leq (2 \log_2n + 4) \Vert R_Y^{-T} E_c R_Y^{-1} \Vert_2.
\end{equation*}
We proceed bounding $\Vert R_Y^{-T} E_c R_Y^{-1} \Vert_2$ as
\begin{align*}
    \Vert R_Y^{-T} E_c R_Y^{-1} \Vert_2 = & \, \Vert Q_Y^T \Delta_H R_Y^{-1}  +  R_Y^{-T} \Delta_H^T Q_Y + R_Y^{-T} \Delta_H^T \Delta_H R_Y^{-1} \Vert_2 \\
    \leq & \, 2 \Vert \Delta_H R_Y^{-1} \Vert_2 + \Vert \Delta_H R_Y^{-1} \Vert_2^2 \\
    \leq & \, \widetilde{\gamma}_{sn}^{(QR)}  \|\Yhat \|_F \Vert R_Y^{-1} \Vert_2 + (\widetilde{\gamma}_{sn}^{(QR)})^2  \|\Yhat \|_F^2 \Vert R_Y^{-1} \Vert_2^2 \\
    \leq & \, n^{1/2} \widetilde{\gamma}_{sn}^{(QR)}  \kappa_2(\Yhat) + n (\widetilde{\gamma}_{sn}^{(QR)})^2  \kappa_2(\Yhat)^2
\end{align*}
and thus
\begin{equation}\label{eq:gamma1_bound}
     \Vert \Gamma_1 \Vert_2 \leq (2 \log_2n + 4) n^{1/2} \widetilde{\gamma}_{sn}^{(QR)}  \kappa_2(\Yhat)\left( 1 + n^{1/2} \widetilde{\gamma}_{sn}^{(QR)}  \kappa_2(\Yhat) \right) \eqqcolon \alpha_1.
\end{equation}
By considering the Cholesky decompositions of $(\Omega A)^T \Omega A$ and $(\Omega A + \Delta_s)^T(\Omega A + \Delta_s)$ and using the same argument as above, we obtain
\begin{align}
    \Vert \Gamma_2 \Vert_2 \leq & \, (2 \log_2n + 4)  \left( n^{1/2} \Vert R^{-1} \Vert_2 \Vert \Delta_s \Vert_2 + n\Vert R^{-1} \Vert_2^2 \Vert \Delta_s \Vert_2^2 \right) \nonumber \\
    = & \, (2 \log_2n + 4)  n^{1/2} \Vert (\Omega A)^{\dagger} \Vert_2 \Vert \Delta_s \Vert_2 \left(1 + n^{1/2} \Vert (\Omega A)^{\dagger} \Vert_2 \Vert \Delta_s \Vert_2 \right)\eqqcolon \alpha_2. \label{eq:gamma2_bound}
\end{align}
To obtain the bounds \eqref{eq:Rhat_inv_norm_bound}-\eqref{eq:ARhat_inv_cond_bound}, we have to consider $\Rhat^{-1}$ and $R_Y^{-1}$. We do this using the first order approximations
\begin{gather*}
    \Rhat^{-1} = \left( (I + \Gamma_1) R_Y \right)^{-1} \approx R_Y^{-1} (I - \Gamma_1)\\
     R_Y^{-1} = \left( (I + \Gamma_2) R \right)^{-1} \approx R^{-1} (I - \Gamma_2),
\end{gather*}
that are valid under assumptions \eqref{eq:assumption_u_qr} and \eqref{eq:assumption_u_s}. From the above approximations, \eqref{eq:Rhat_R_Y}, and \eqref{eq:R_Y_R}, we have
\begin{gather*}
    \Rhat = (I + \Gamma_1) (I + \Gamma_2) R \quad \text{ and} \\
    \Rhat^{-1} \approx R^{-1} (I - \Gamma_2) (I - \Gamma_1).  
\end{gather*}
Note that $(A\Rhat^{-1})^{\dagger} = \Rhat A^{\dagger}$, because $A$ is full rank and $\Rhat$ is invertible.
Taking the norms, using
\begin{gather*}
    \Vert (I + \Gamma_1) (I + \Gamma_2) \Vert_2 \leq (1 + \Vert \Gamma_1 \Vert_2 ) (1 + \Vert \Gamma_2 \Vert_2) \leq (1+\alpha_1)(1+\alpha_2), \\
    \Vert (I - \Gamma_2) (I - \Gamma_1) \Vert_2 \leq (1 + \Vert \Gamma_1 \Vert_2 ) (1 + \Vert \Gamma_2 \Vert_2) \leq (1+\alpha_1)(1+\alpha_2),
\end{gather*}
noting that under the assumptions \eqref{eq:assumption_u_qr} and \eqref{eq:assumption_u_s}
\begin{align*}
\alpha_1 & \leq 2 (2 \log_2n + 4) n^{1/2} \widetilde{\gamma}_{sn}^{(QR)}  \kappa_2(\Yhat), \\
    \alpha_2 & \leq 2(2 \log_2n + 4)  n^{1/2} \Vert (\Omega A)^{\dagger} \Vert_2 \Vert \Delta_s \Vert_2,
\end{align*}
and substituting $\kappa_2(\Yhat)$ in the bound for $\alpha_1$ by
\begin{equation*}
    \kappa_2(\Yhat) = \kappa_2(R_Y) \leq \kappa_2(R) (1 + \Vert \Gamma_2 \Vert_2)^2 \leq \kappa_2(R) (1 +  \alpha_2)^2, 
\end{equation*}
where using \eqref{eq:assumption_u_s} we have $ (1 +  \alpha_2)^2 \leq 9$, and using
\begin{equation*}
    \Vert R \Vert_2 = \Vert \Omega A \Vert_2 \quad \text{and} \quad  \Vert R^{-1} \Vert_2 = \Vert (\Omega A)^{\dagger} \Vert_2
\end{equation*}
gives the required results.
\end{proof}

\subsection{Comments}\label{sec:comments_on_theorem}
We now comment on the assumptions and results of Theorem~\ref{th:finite_recision_error_R}. In \eqref{eq:assumption_u_qr}, we assume that the computed sketched matrix is not severely ill-conditioned in the precision which is used to compute its QR decomposition; this is a standard assumption, which also guides against using $u_{QR} \gg u_s$. 
In \eqref{eq:assumption_u_s} we assume that the sketching error is small enough to neutralise $\Vert (\Omega A)^{\dagger} \Vert_2 $. Note that both quantities here depend on the type of sketching used. The results show that when the sketching matrix and precisions are chosen carefully, the norms and condition numbers of $\Rhat$ and $A \Rhat^{-1}$ are close to the norms and condition numbers when using the exact $R$ factor of $\Omega A$. 

We can specify the results when the sketch is computed as a matrix-matrix product and thus 
\begin{equation*}
    \Vert \Delta_s \Vert_2 \leq u_s \Vert A \Vert_2  \Vert \Omega \Vert_2 + m^{1/2} \gamma_m^{(s)} \Vert A \Vert_2  \Vert \Omega \Vert_2,
\end{equation*}
where the first term is due to casting $A$ into $u_s$ and the second term comes from the matrix-matrix multiplication. Combining this with  \eqref{eq:assumption_u_s} gives condition
\begin{equation*}
     (2 \log_2n + 4) n^{1/2} m^{1/2} \gamma_{m+1}^{(s)} \frac{\Vert A \Vert_2  \Vert \Omega \Vert_2}{\Vert \Omega A \Vert_2} \kappa_2(\Omega A)  < 1. 
\end{equation*}
We thus require the exact sketched matrix $\Omega A$ to be not too ill-conditioned in the sketching precision $u_s$. The bound for $\beta$ is then
\begin{equation*}
    \beta \leq  1 + 2(2 \log_2n + 4) n^{1/2} \left( 9\widetilde{\gamma}_{sn}^{(QR)} + m^{1/2} \gamma_{m+1}^{(s)} \frac{\Vert A \Vert_2  \Vert \Omega \Vert_2}{\Vert \Omega A \Vert_2} \right)  \kappa_2(\Omega A) + \mathcal{O}(u_{QR}u_{s}).
\end{equation*}
A mixed precision implementation of Algorithm~\ref{alg:generating_R} in \cite{georgiou2023mixed} uses $u_{QR} \leq u_s$, that is, the QR decomposition is computed in precision higher than in the sketching step. Note that our results do not suggest a higher quality approximation in this setting.

We note that the term $\beta$ can be seen as an ``amplification'' factor which determines the ratio between the computed and exact quantities.  
In the case that $u_s=u_{QR}=0$ (i.e., we compute in exact arithmetic), then $\beta=1$ and thus the computed quantities are the same as the exact quantities.

If $u_s$ or $u_{QR}$ is set to a low precision with narrow range, for example, IEEE half precision, then scaling may be needed to avoid underflow and overflow; see \cite{higham2019squeezing} for a deeper discussion. A one-sided diagonal scaling strategy is presented in \cite[lines 1-3 of Algorithm 3.1]{carson2017new}. Here one constructs a diagonal matrix $S$ that contains the reciprocals of the modulus largest elements of each column of $A$ with a positive sign and $AS$ is computed, that is each column of $A$ is divided by its modulus largest value to avoid overflow. Then $AS$ is multiplied by a positive parameter to increase the values within the range of the low precision and avoid underflow. $AS$ can be used instead of $A$ in Algorithm~\ref{alg:generating_R} and the preconditioner is then set to $RS^{-1}$.

\subsection{Previous work}
The errors in uniform precision versions of Algorithm~\ref{alg:generating_R} have been analysed in \cite{meier2024sketch,epperly2024fast}, and \cite{higgins2023analysis}. The first manuscript provides deterministic bounds for $\kappa_2(\Rhat)$ and $\kappa_2(A\Rhat^{-1})$ in terms of, respectively, $\kappa_2(\Omega Q_A) \kappa_2(A)$ and $\kappa_2(\Omega Q_A)$, where $A = Q_A R_A$ is a truncated QR factorization. The second one provides bounds in terms of $\kappa_2(A)$ and the distortion of the subspace embedding. The final listed work gives probabilistic bounds for $\Vert \Rhat^{-1} \Vert_2$ and $\Vert A \Rhat^{-1} \Vert_2$ when using multisketching.

\subsection{Sketching ill-conditioned matrices in low precision}\label{sec:sketching_regularized}
If we choose $\Omega$ to be a subspace embedding, then $\kappa_2(\Omega A)$ is close to $\kappa_2(A)$ and assumption \eqref{eq:assumption_u_s} essentially limits the applicability of the analysis to cases where $\kappa_2(A) < u_s^{-1}$. Thus, $u_s$ can be set to half precision only if $\kappa_2(A) < 2^{11}=2048$. We can alternatively use Theorem~\ref{th:finite_recision_error_R} with $A$ replaced by its full-rank low precision version $A_s$ and obtain the result in the following corollary.
\begin{cor}
Let $A \in \mathbb{R}^{m \times n}$ be full rank and $A_s$ denote $A$ cast to a precision $u_s$, such that 
\begin{gather*}
    A_s = A + E, \quad \textrm{where } \Vert E \Vert_2 \leq \sqrt{n} u_s \Vert A \Vert_2, \quad \textrm{and} \\
    \Omega A_s = Q_s R_s
\end{gather*}
is the economic QR decomposition. We assume that $A_s$ is full-rank and assume that the assumptions of Theorem~\ref{th:finite_recision_error_R} hold with $A$ replaced by $A_s$, and denote the resulting $\beta$ as $\beta_s$. 
Then \eqref{eq:Rhat_norm_bound} - \eqref{eq:ARhat_inv_cond_bound} can be written as 
 \begin{align*}
        \Vert \Rhat \Vert_2 & \leq \beta_s \Vert \Omega A_s \Vert_2, \\
        \Vert \Rhat^{-1} \Vert_2 & \leq \beta_s \Vert (\Omega A_s)^{\dagger} \Vert_2, \\
        \kappa_2(\Rhat) & \leq \beta_s^2 \kappa_2(\Omega A_s),\\
        \Vert A \Rhat^{-1} \Vert_2 & \leq \beta_s \Vert AR_s^{-1} \Vert_2, \\
         \Vert (A \Rhat^{-1})^{\dagger} \Vert_2 & \leq \beta_s \Vert (AR_s^{-1})^{\dagger} \Vert_2,\\
         \kappa_2(A \Rhat^{-1}) & \leq \beta_s^2 \kappa_2( A R_s^{-1}).
    \end{align*}
\end{cor}

$\Rhat$ is thus close to $R_s$. The relative perturbation to the largest singular values of $A$ coming from casting it to the lower precision $u_s$ is expected to be small and we thus expect $\Vert \Omega A_s \Vert_2 \approx \Vert \Omega A \Vert_2$. The small singular values of a tall and skinny $A_s$ can, however, be significantly larger than the small singular values of $A$ when $\kappa_2(A) > u_s^{-1}$, that is, the low precision has a \textit{regularizing} effect on $A$; see, e.g., \cite{boutsikas2023small} for deterministic results and \cite{dexter2024stochastic} for the stochastic rounding case. A large increase in the smallest singular values would give $\Vert A_s^{\dagger} \Vert_2 \ll \Vert A^{\dagger} \Vert_2$ and thus $\kappa_2(A_s) \ll \kappa_2(A)$. This would allow the analysis to be applied to a wider range of problems, namely, when $\kappa_2(A_s) < u_s^{-1}$ and casting to lower precision preserves the rank. 

In order to use the bounds involving $A_s$ in the analysis of IR, we need to upper bound them by terms involving $A$. We can bound $\Vert \Omega A_s \Vert_2$ as
\begin{equation*}
    \Vert \Omega A_s \Vert_2 = \Vert \Omega (A + E) \Vert_2 \leq \Vert \Omega A \Vert_2 + \sqrt{n} u_s \Vert \Omega \Vert_2 \Vert A \Vert_2,
\end{equation*}
and using the assumption that $A$ is full rank and thus has linearly independent columns while $R_s$ has linearly independent rows we bound $\Vert (A R_s^{-1})^{\dagger} \Vert_2$ as
\begin{align*}
    \Vert (AR_s^{-1})^{\dagger} \Vert_2 = & \, \Vert R_s A^{\dagger} \Vert_2 \\
                                        = & \, \Vert Q_s^T \Omega (A+E) A^{\dagger} \Vert_2 \\ 
                                        \leq & \, \Vert Q_s^T \Omega A A^{\dagger} \Vert_2 + \Vert  Q_s^T \Omega E A^{\dagger} \Vert_2 \\
                                         \leq & \, \Vert \Omega \Vert_2 + \Vert \Omega E  A^{\dagger} \Vert_2 \\
                                          \leq & \, \left( 1 + \sqrt{n} u_s \Vert A \Vert_2 \Vert  A^{\dagger} \Vert_2 \right) \Vert \Omega \Vert_2 \\
                                          = & \,\left( 1 + \sqrt{n} u_s \kappa_2(A) \right) \Vert \Omega \Vert_2,
\end{align*}     
which we observe to be descriptive in our numerical experiments. Obtaining useful bounds for $\Vert (\Omega A_s)^{\dagger} \Vert_2$ and $\Vert A R_s^{-1} \Vert_2$ however proves challenging and our numerical experiments not reported here suggest that it depends on the ratio $n/m$. 
We illustrate the behavior of $\Vert (A R_s^{-1})^{\dagger} \Vert_2$ in Figure~\ref{fig:psinverse_norm_casted_preconditioned} with a small MATLAB example. In order to focus on the effect of casting to lower precision, no sketching is used, that is, $\Omega$ is set to an identity matrix. We cast $A$ to single and half precisions, compute the R factors of the cast matrices in double precision and use these to precondition $A$. 

\begin{figure}
\centering
        \includegraphics[width=0.7\textwidth]{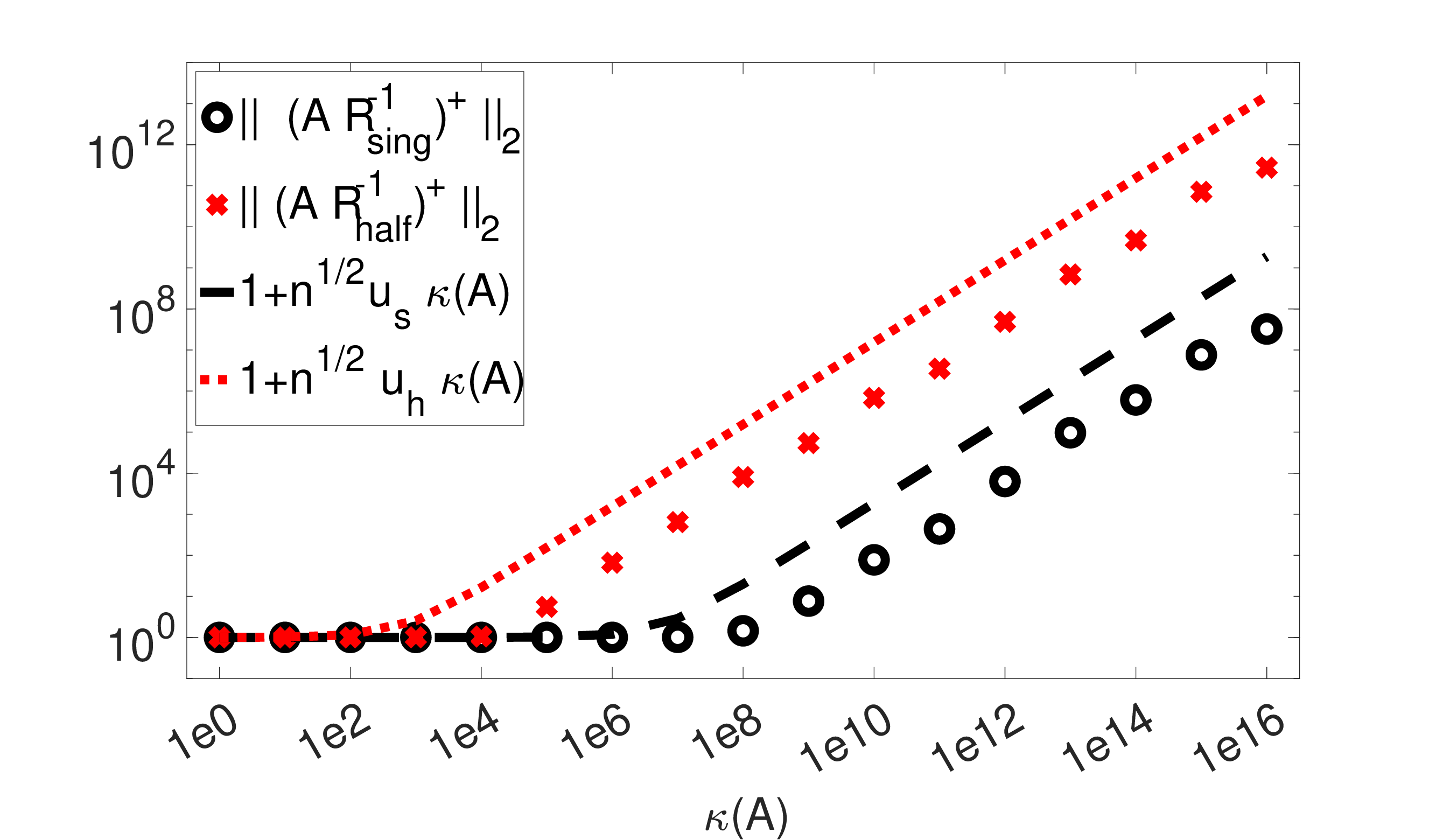}
    \caption{$\Vert (AR_{sing}^{-1})^{\dagger} \Vert_2$ and $\Vert (AR_{half}^{-1})^{\dagger} \Vert_2$, when $A$ is generated in double precision as $gallery('randsvd',[400,10],\kappa(A),3)$, and factors $R_{sing}$ and $R_{half}$ are obtained by casting $A$ to single and half precisions, respectively, and computing economic QR decompositions of the lower precision matrices. }
\label{fig:psinverse_norm_casted_preconditioned}
\end{figure}

\section{Sketch-and-precondition FGMRES-LSIR}\label{sec:lsir}
We now consider how the \newline randomized preconditioner can be used in the LSIR setting and provide theoretical convergence guarantees. Recall that in previous work \cite{carson2020three} it was shown that using preconditioned GMRES in LSIR allows solving more ill-conditioned problems than when using Bj{\"o}rck's approach that employs the QR factorization of $A$ \cite{bjorck1967solving}. However the theoretical convergence guarantees in \cite{carson2020three} hold only for a particular left-preconditioner using the full QR factors, which can be expensive to compute and apply in practice. We extend this work to account for a preconditioner that employs the randomized $R$-factor only (i.e., the $Q$ factor is not needed) and is applied as a split-preconditioner in FGMRES instead of a left-preconditioner in GMRES.

We describe this LSIR procedure
in Algorithm~\ref{alg:augmented_approach}. LSQR is initialized with the so-called sketch-and-solve solution $x = R^{-1}(Q^T \Omega b)$, where $\Omega A = QR$; such an initialization for a least-squares solver was originally proposed in \cite{rokhlin2008tygert} and it has been observed that it can significantly improve the accuracy of the final solution \cite{meier2024sketch,epperly2024fast}. LSQR is run in the working precision with unit roundoff $u$, which is also used as the working precision in IR. In the IR loop, the residuals of the augmented system are computed in precision with unit roundoff $u_r$ and we use this precision for a triangular solve with $R$ to obtain the preconditioned right-hand side. We employ a mixed precision FGMRES variant, where the applications of $M_L^{-1}$, $M_R^{-1}$, and $\Atld$ to a vector are computed in precisions with unit roundoffs $u_L$, $u_R$, and $u_A$, respectively, and other computations are performed in the working precision $u$ \cite{carson2024stability}; $M_L^{-1}$, $M_R^{-1}$, and $\Atld$ are as defined in \eqref{eq:ir_system_split_prec}.

\begin{algorithm}
\caption{Augmented system LSIR with randomized preconditioning}\label{alg:augmented_approach}
\algorithmicrequire  $A \in \mathbb{R}^{m \times n}$, $b \in \mathbb{R}^m$, $R$ computed via Algorithm~\ref{alg:generating_R}, $x^{s} = R^{-1}(Q^T \Omega b)$, IR precisions $u_r$ and $u$ where $u_r \leq u$, FGMRES precisions $u_A$, $u_L$, $u_R$  \\
\algorithmicensure approximate solution $x$
\begin{algorithmic}[1]
\State Solve $ x_0 = arg \min_x \Vert b-Ax\Vert_2$ via LSQR initialised with $x^{s}$ and right-preconditioned with $R$ \Comment{$u$ }
\State Compute $r_0 = b - A x_0$ \Comment{$u$ }
\State $i=0$
\WHILE{not converged} \label{irstep:while_start}
\State Compute    $
        \begin{bmatrix}
 f_i \\g_i
\end{bmatrix}
= \begin{bmatrix}
 b \\0
\end{bmatrix}
-
\begin{bmatrix}
 I & A\\ A^T & 0
\end{bmatrix}
\begin{bmatrix}
r_i \\  x_i
\end{bmatrix}
$ \Comment{$u_r$} 
\State Compute $h_i = R^{-T} g_i$ via triangular solve  \Comment{$u_r$}  \label{irstep:res_comp2}
\State Solve via split-preconditioned FGMRES 
\begin{equation*}
\begin{bmatrix}
 I & 0 \\ 0 & R^{-T}
\end{bmatrix}
\begin{bmatrix}
 I & A \\ A^T & 0
\end{bmatrix}
\begin{bmatrix}
 I & 0 \\ 0 & R^{-1}
\end{bmatrix}
\begin{bmatrix}
\delta r_i \\ \delta z_i
\end{bmatrix}
=
\begin{bmatrix}
 f_i \\ h_i
\end{bmatrix},
\end{equation*} 
where 
\begin{equation*}
    \begin{bmatrix}
 I & 0 \\ 0 & R
\end{bmatrix} \begin{bmatrix}
\delta r_i \\ \delta x_i
\end{bmatrix}
=\begin{bmatrix}
\delta r_i \\ \delta z_i
\end{bmatrix}
\end{equation*} \Comment{$u$, $u_A$, $u_L$, $u_R$} \label{irstep:fgmres_solve}
\State Update $
\begin{bmatrix}
r_{i+1} \\  x_{i+1} 
\end{bmatrix} = 
\begin{bmatrix}
r_i \\  x_i 
\end{bmatrix}
+ 
\begin{bmatrix}
\delta r_i \\ \delta x_i 
\end{bmatrix} $ \Comment{$u$} \label{irstep:update}
\State $i=i+1$
\ENDWHILE \label{irstep:while_end}
\end{algorithmic}
\end{algorithm}

We further explore theoretical convergence guarantees for this LSIR approach. This requires bounds for the condition number of the preconditioned system, which we obtain in the next subsection. The following notation is used.
\begin{gather}\label{eq:IR_notation}
    \btld = \begin{bmatrix}
        b \\ 0
    \end{bmatrix}, \
    d = \begin{bmatrix}
        r \\ x
    \end{bmatrix}, \
    d_i = \begin{bmatrix}
        r_i \\ x_i
    \end{bmatrix}, \
     \ddelta_i = \begin{bmatrix}
        \rdelta_i \\ \xdelta_i
    \end{bmatrix}, \
    y_i = \begin{bmatrix}
        \rdelta_i \\ \delta z_i
    \end{bmatrix}, \
    w_i = \begin{bmatrix}
        f_i \\ g_i
    \end{bmatrix}, \ 
    s_i = \begin{bmatrix}
        f_i \\ h_i
    \end{bmatrix}.
\end{gather}

\subsection{Condition number of the preconditioned augmented matrix}
We consider the preconditioned coefficient matrix in \eqref{eq:ir_system_split_prec} and obtain two bounds for it: one in terms of $\Vert A \Rhat^{-1} \Vert_2$ and $\Vert ( A\Rhat^{-1})^{\dagger}\Vert_2$, and another in terms of $\kappa(A\Rhat^{-1} )$.

We use Bj\"{o}rck's approach \cite{bjorck1967solving} for bounding $\kappa_2(\Atld)$ to bound \linebreak $\kappa_2 (M_L^{-1} \Atld M_R^{-1})$. Consider the following scaled coefficient matrix
\begin{equation*}
    M_L^{-1} \Atld_{\alpha} M_R^{-1} \coloneqq \begin{bmatrix}
        \alpha I & A \Rhat^{-1} \\ \Rhat^{-T} A^T & 0
    \end{bmatrix},
\end{equation*}
where $\alpha > 0$. The condition number of the preconditioned matrix is
\begin{equation*}
    \kappa(M_L^{-1} \Atld_{\alpha} M_R^{-1}) = \frac{\alpha + \sqrt{\alpha^2 + 4 \sigma_{max}(A\Rhat^{-1})^2}}{\min \{2, \sqrt{\alpha^2 + 4 \sigma_{min}(A\Rhat^{-1})^2} - \alpha \} }.
\end{equation*}
If no scaling is used, that is, $\alpha=1$, then we have
\begin{align*}
    \kappa(M_L^{-1} \Atld_{\alpha} M_R^{-1}) & = \frac{1 + \sqrt{1 + 4 \sigma_{max}(A\Rhat^{-1})^2}}{\min \{2, \sqrt{1 + 4 \sigma_{min}(A\Rhat^{-1})^2} - 1 \} } \\
    & \leq \frac{2 + 2 \Vert A\Rhat^{-1}\Vert_2}{\min \{2, \sqrt{1 + 4 / \Vert ( A\Rhat^{-1})^{\dagger}\Vert^2} - 1 \} }, 
\end{align*}
where we use $\sqrt{a^2 + b^2} \leq a +b$ when $a,b>0$. Thus
\begin{equation}\label{eq:precond_augmented_condition_bound_no_scaling}
\kappa(M_L^{-1} \Atld_{\alpha} M_R^{-1}) \leq
    \begin{cases}
        1 + \Vert A\Rhat^{-1}\Vert_2,  & \text{ if } \Vert ( A\Rhat^{-1})^{\dagger}\Vert_2 \leq 1/\sqrt{2}, \\
        \frac{2 + 2 \Vert A\Rhat^{-1}\Vert_2}{\sqrt{1 + 4 / \Vert ( A\Rhat^{-1})^{\dagger}\Vert_2^2} - 1  } & \text{otherwise. } 
    \end{cases}
\end{equation}
The optimal scaling by $\alpha = 2^{-1/2} \sigma_{min}(A \Rhat^{-1})$ gives 
\begin{equation}\label{eq:L2cond_Bjorks}
    \kappa(M_L^{-1} \Atld_{\alpha} M_R^{-1}) \leq 2 \kappa(A \Rhat^{-1}).
\end{equation}
This scaling is, however, expensive to compute. 
If there is no scaling and \linebreak 
$\Vert ( A\Rhat^{-1})^{\dagger}\Vert_2 \gg 1$, then $\kappa(M_L^{-1} \Atld_{\alpha} M_R^{-1})$ can be close to $\kappa(A\Rhat^{-1})^2$. Note that when $\Rhat$ is computed via Algorithm~\ref{alg:generating_R} and assumptions of Theorem~\ref{th:finite_recision_error_R} hold with $\Omega$ chosen as a subspace embedding, we do \textit{not} expect $\Vert ( A\Rhat^{-1})^{\dagger}\Vert_2 \gg 1$ to hold. In this case we also expect $\kappa(A\Rhat^{-1})$ to be small. If however $\kappa_2(A) > u_s$, then $\Vert ( A\Rhat^{-1})^{\dagger}\Vert_2$ and as a result $\kappa(A \Rhat^{-1})$ can grow substantially as we have discussed in Section~\ref{sec:sketching_regularized}. Then our bounds above show that $\kappa(M_L^{-1} \Atld_{\alpha} M_R^{-1})$ can be ill-conditioned. 

We thus have two bounds, where one depends on the norm of the preconditioned least-squares coefficient matrix and its pseudoinverse ($\Atld$ is not scaled), and the other on its condition number ($\Atld$ is scaled). We consider a simple numerical example and show $\kappa(M_L^{-1} \Atld M_R^{-1})$ and bounds \eqref{eq:precond_augmented_condition_bound_no_scaling} and \eqref{eq:L2cond_Bjorks} in Figure~\ref{fig:cond_no_augmented_bounds}. Note that if $\kappa_2(A) < u_s^{-1}$ then $\kappa(M_L^{-1} \Atld M_R^{-1})$ stays close to \eqref{eq:L2cond_Bjorks} even without scaling and \eqref{eq:precond_augmented_condition_bound_no_scaling} is a tight bound; if the condition is violated, the perturbation due to sketching in low precision results in a poor preconditioner and thus scaling becomes important and $\kappa(M_L^{-1} \Atld M_R^{-1})$ grows large. Note that in this example with $u_s$ set to single and $\kappa_2(A)=10^{16}$ we cannot compute \eqref{eq:precond_augmented_condition_bound_no_scaling} in double precision, because $\sqrt{1 + 4 / \Vert ( A\Rhat^{-1})^{\dagger}\Vert_2^2}$ is evaluated as 1. 

\begin{figure}
    \centering
    \begin{subfigure}[t]{0.48\linewidth}
        \includegraphics[width=\textwidth]{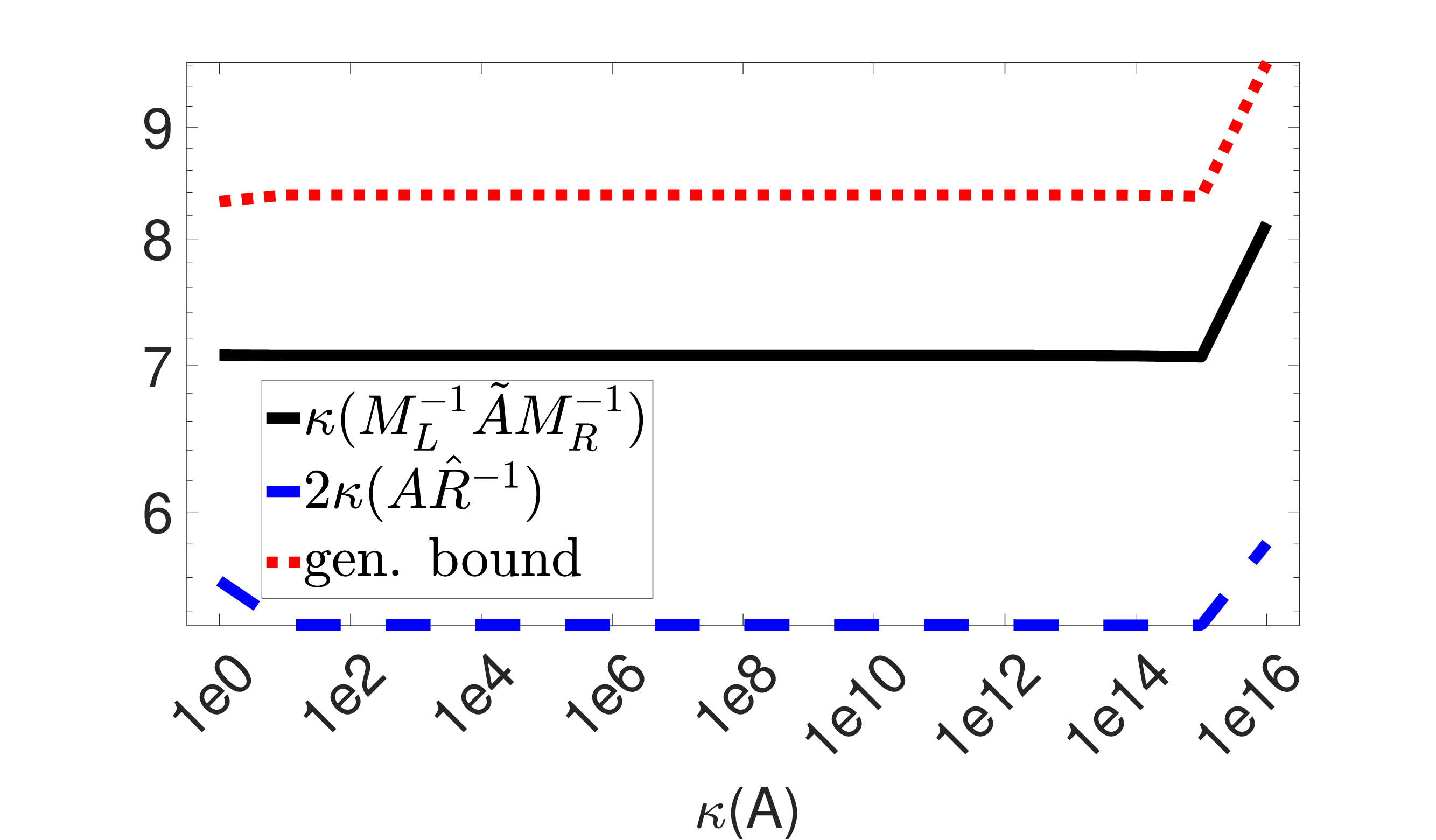}
    \end{subfigure}
    \begin{subfigure}[t]{0.48\linewidth}
        \includegraphics[width=\textwidth]{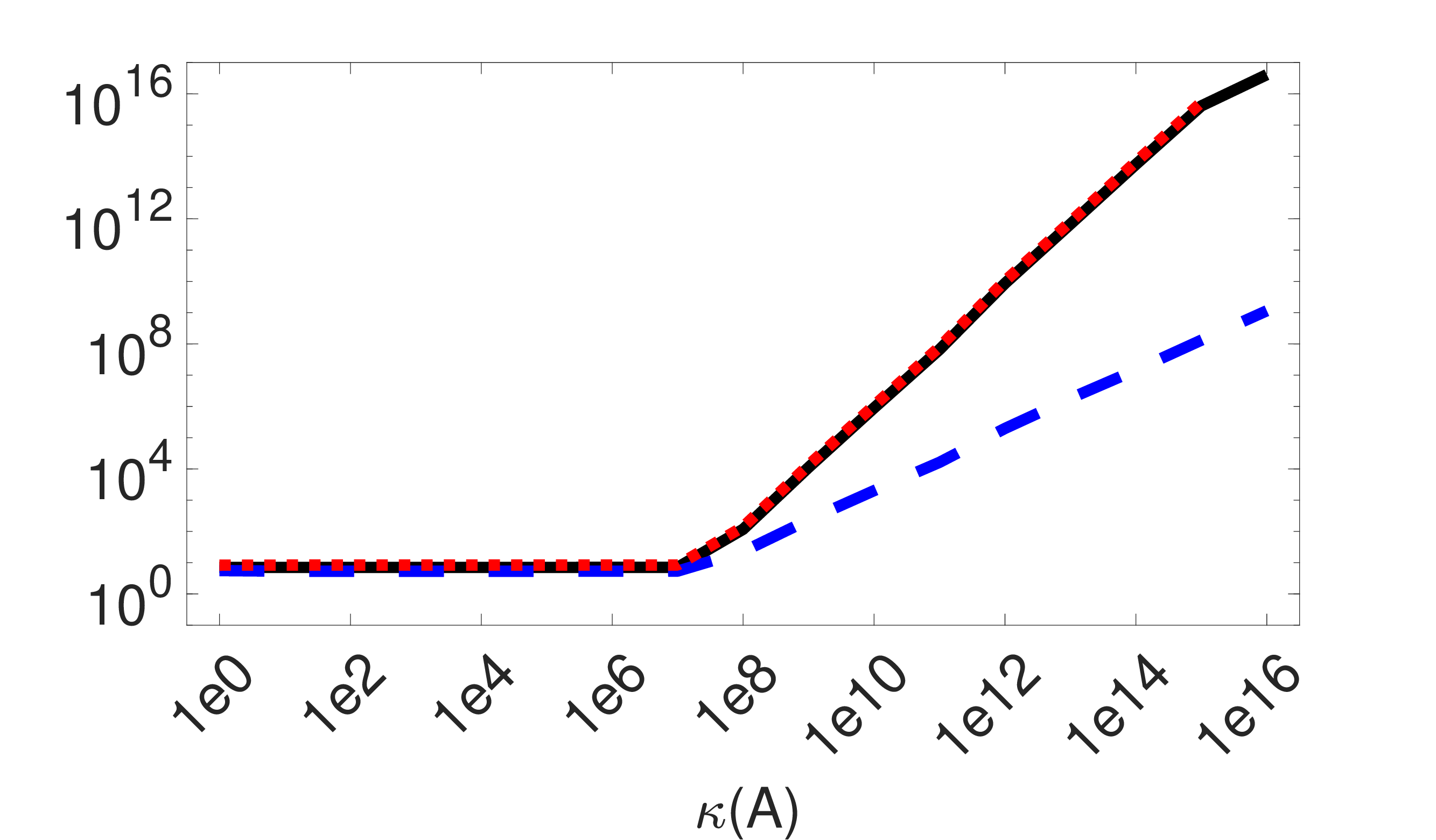}
    \end{subfigure}
    \caption{Condition number of the preconditioned augmented system without scaling and bounds \eqref{eq:precond_augmented_condition_bound_no_scaling} (gen. bounds) and \eqref{eq:L2cond_Bjorks} when $u_s$ is set to double (left panel) and single (right panel); $u_{QR}$ is set to double. $A$ is generated as in Figure~\ref{fig:psinverse_norm_casted_preconditioned} and $\Omega = \frac{1}{\sqrt{4n}}G$, where $G$ is a $4n \times m$ random matrix with Gaussian entries.}
    \label{fig:cond_no_augmented_bounds}
\end{figure}

\subsection{IR convergence guarantees} 
The aim of the preconditioner is to ensure the convergence of FGMRES-IR. We provide the analysis first and summarize it in a theorem at the end of the section. Carson and Higham \cite{carson2018accelerating} proved that IR for solving linear systems of equations converges under some conditions on the computed solution in the refinement steps. Namely, the forward error $\frac{\Vert d - \dhat_i \Vert}{\Vert d \Vert}$ is guaranteed to converge to the limiting accuracy $4pu_r\textrm{cond}(\Atld,d)+ u$, where $p$ is the maximum number of nonzeros per row of $\begin{pmatrix} \Atld & d \end{pmatrix}$ and $\textrm{cond}(\Atld,d) = \frac{\Vert \vert \Atld^{-1} \vert \vert \Atld \vert \vert d \vert \Vert}{\Vert d \Vert }$, if the computed updates $\ddeltah_i = \begin{bmatrix} \delta \rhat_{i}^T &  \delta \xhat_{i}^T  \end{bmatrix}^T$ satisfy a bound on the relative normwise forward error
\begin{equation}\label{eq:forward_error_IRcondition}
    \frac{\Vert  \ddelta_i - \ddeltah_i \Vert }{\Vert \ddelta_i \Vert} = u_g \Vert E_i \Vert < 1.
\end{equation}
The normwise relative backward error $\frac{\Vert \btld - \Atld \dhat \Vert_2}{\Vert \btld \Vert_2 + \Vert \Atld \Vert_2 \Vert \dhat \Vert_2}$ is guaranteed to converge to $pu$ if 
\begin{equation}\label{eq:backward_error_IRcondition}
    \left( c_1 \kappa(\Atld) + c_2 \right)u_g < 1,
\end{equation}
is satisfied, where
\begin{equation*}
    \frac{\Vert \what_i  - \Atld \ddeltah_i \Vert  }{ c_1 \Vert \Atld \Vert \Vert \ddeltah_i \Vert + c_2 \Vert \what_i \Vert} \leq u_g,
\end{equation*}
and $E_i$, $c_1$, and $c_2$ are functions of $\Atld$, $m+n$, $u_g$, and $\what_i$ with nonnegative entries. We thus need to determine values of $u_g$, $\Vert E_i \Vert$, $c_1$, and $c_2$ to determine for what $\kappa(A)$ we can expect FGMRES-IR (Algorithm~\ref{alg:augmented_approach}) to converge.

We tackle this using bounds on the forward and backward error of the mixed-precision FGMRES variant.
As shown in \cite[eq. (2.12)]{carson2024stability}, the relative normwise forward error is bounded by
\begin{equation*}\label{eq:forward_error_mpfgmres}
     \frac{\Vert \ddelta_i - \ddeltah_i \Vert_2}{\Vert \ddelta_i \Vert_2} \leq  \frac{1.3c(n,k)}{1 - \rho} \zeta \kappa_2(M_L^{-1} \Atld M_R^{-1}) \kappa_2(M_R)
\end{equation*}
and the relative normwise backward error of the augmented system (note that this is not the backward error of the least-squares problem) is bounded by \cite[Corollary 2.2]{carson2024stability}
\begin{equation*}
   \frac{\Vert \what_i  - \Atld \ddeltah_i \Vert_2  }{ \Vert \Atld \Vert_2 \Vert \ddeltah_i \Vert_2 +  \Vert \what_i \Vert_2} \leq   \frac{1.3 c(n,k)}{1 - \rho} \zeta \kappa_2(M_L),
\end{equation*}
where the second order terms are ignored, $c(n,k)$ is a constant depending on $n$ and $k$, $\rho < 1$ depends on $n$, the number of FGMRES iterations $k$, $u$, $u_R$, $M_R$, and $\Zhat_k$ (the preconditioned basis matrix arising in FGMRES), and $\zeta$ depends on $u$, $u_A$, $u_L$, $M_L$, $M_R$, $\Atld$, $\Zhat_k$, $\ddeltah_i$, and $\shat_i$; see \cite{carson2024stability} for further details. We note that $\zeta$ is expected to dominate the term $\frac{1.3 c(n,k)}{1 - \rho} \zeta $, and the achievable backward error depends on the interaction between the precisions $u$, $u_A$, and $u_L$, and the error in applying $M_L^{-1}$ and computing the matrix-vector products with $\Atld$; guidance on how to achieve $\frac{1.3 c(n,k)}{1 - \rho} \zeta = \mathcal{O}(u)$ is provided in \cite[Section 2.2]{carson2024stability}. 

Since in our case $\kappa_2(M_L) = \kappa_2(M_R)$, we can set 
\begin{gather*}
    u_g = \frac{1.3 c(n,k)}{1 - \rho} \zeta \kappa_2(M_L),\\
    \Vert E_i \Vert \leq \kappa_2(M_L^{-1} \Atld M_R^{-1}) \quad \textrm{ and } \\
    c_1 = c_2 = 1.
\end{gather*}
Then \eqref{eq:forward_error_IRcondition} holds if
\begin{equation}\label{eq:fe_IR_cond_2}
     \frac{1.3 c(n,k)}{1 - \rho} \zeta \kappa_2(M_L) \kappa_2(M_L^{-1} \Atld M_R^{-1})< 1
\end{equation}
and \eqref{eq:backward_error_IRcondition} requires
\begin{equation}\label{eq:be_IR_cond_2}
    \frac{1.3 c(n,k)}{1 - \rho} \zeta \kappa_2(M_L) \kappa(\Atld) <1.
\end{equation}
We assume in the following that the precisions for matrix-vector products with $\Atld$ and applying the preconditioners $M_L$ and $M_R$ in FGMRES are set so that 
\begin{equation}\label{eq:fgmres_precisions_condition}
     \frac{1.3 c(n,k)}{1 - \rho} \zeta = \mathcal{O}(u).
\end{equation}
Then ignoring the constants, \eqref{eq:fe_IR_cond_2} gives the following condition for the forward error to converge: 
\begin{equation}\label{eq:fe_IR_cond_interim}
   \kappa_2(M_L) \kappa_2(M_L^{-1} \Atld M_R^{-1}) < \mathcal{O}(u^{-1}).
\end{equation}
Note that in our case
\begin{equation*}
    \kappa_2(M_L) = \kappa_2(M_R) = \max \{1, \Vert \Rhat \Vert_2 \} \max \{1,  \Vert \Rhat^{-1} \Vert_2\} = \max \{ \Vert \Rhat \Vert_2, \Vert \Rhat^{-1} \Vert_2, \kappa(\Rhat) \}
\end{equation*}
and combining this with \eqref{eq:Rhat_norm_bound} - \eqref{eq:Rhat_cond_bound} gives
\begin{equation}\label{eq:kappa_ML_bound}
    \kappa_2(M_L) \leq \beta \max \{ \Vert \Omega A \Vert_2, \Vert (\Omega A)^{\dagger} \Vert_2, \beta \kappa_2(\Omega A) \} \eqqcolon \psi(\beta,\Omega A).
\end{equation}
Recall that we have have two options for bounding $\kappa_2(M_L^{-1} \Atld M_R^{-1})$. Using \eqref{eq:precond_augmented_condition_bound_no_scaling}, \eqref{eq:ARhat_inv_norm_bound}, and \eqref{eq:ARhat_inv_pinv_norm_bound} we obtain 
\begin{equation}
\kappa(M_L^{-1} \Atld_{\alpha} M_R^{-1}) \leq
    \begin{cases}
        1 + \beta \Vert AR^{-1}\Vert_2,  & \text{ if } \Vert ( A\Rhat^{-1})^{\dagger}\Vert_2 \leq 1/\sqrt{2}, \\
        \frac{2 + 2 \beta \Vert AR^{-1}\Vert_2}{\sqrt{1 + 4 / \beta \Vert ( AR^{-1})^{\dagger}\Vert_2^2} - 1  } & \text{otherwise } 
    \end{cases}
\end{equation}
and then ignoring the constants and combining this with \eqref{eq:fe_IR_cond_interim} and \eqref{eq:kappa_ML_bound} gives the following condition for the forward error to converge:
\begin{equation*}
   \beta    \max \left\{1, \left(\sqrt{1 + 4 \beta^{-1} \Vert ( AR^{-1})^{\dagger}\Vert_2^{-2}} - 1 \right)^{-1} \right\} \Vert AR^{-1}\Vert_2 \psi(\beta,\Omega A) < u^{-1}.
\end{equation*}
If we apply the optimal scaling $\alpha$, then \eqref{eq:L2cond_Bjorks}, \eqref{eq:ARhat_inv_cond_bound},\eqref{eq:fe_IR_cond_interim}, and \eqref{eq:kappa_ML_bound} give
\begin{equation*}
\beta^2 \kappa_2(A R^{-1}) \psi(\beta,\Omega A) < u^{-1}.
\end{equation*}
Regarding the backward error, the convergence criteria is 
\begin{equation*}
   \psi(\beta,\Omega A) \kappa(\Atld) < u^{-1}.
\end{equation*}
We summarize these results in a theorem below and then comment on them. 
\begin{theorem}
    Assume that the least-squares problem \eqref{eq:LS_problem} is solved via Algorithm~\ref{alg:augmented_approach}, the assumptions of Theorem~\ref{th:finite_recision_error_R} are satisfied, and the precisions in FGMRES are set so that \eqref{eq:fgmres_precisions_condition} holds and $\psi(\beta,\Omega A)$ is as in \eqref{eq:kappa_ML_bound}. 
    Then the relative forward error of the augmented system reaches the limiting value of $4pu_r\textrm{cond}(\Atld,d)+ u$ if
    \begin{equation}\label{eq:fe_converg_bound1}
 \beta    \max \left\{1, \left(\sqrt{1 + 4 \beta^{-1} \Vert ( AR^{-1})^{\dagger}\Vert_2^{-2}} - 1 \right)^{-1} \right\} \Vert AR^{-1}\Vert_2 \psi(\beta,\Omega A) < u^{-1}.
\end{equation}
If the optimal scaling $\alpha = 2^{-1/2} \sigma_{min}(A \Rhat^{-1})$ is applied in step~\ref{irstep:fgmres_solve} of Algorithm~\ref{alg:augmented_approach}, then the forward error convergence condition can be replaced by
\begin{equation}\label{eq:fe_converg_bound2}
\beta^2 \kappa_2(A R^{-1}) \psi(\beta,\Omega A)  < u^{-1}.
\end{equation}
The normwise relative backward error of the augmented system reaches the limiting value $pu$ if
\begin{equation}\label{eq:be_IR_cond_interim}
\psi(\beta,\Omega A) \kappa(\Atld) < u^{-1}.
\end{equation}
\end{theorem}

Note that even if optimal scaling for $\Atld$ is applied and $\kappa(\Atld) \approx \kappa(A)$, the bound for the backward error is more restrictive than the forward error bounds. Note however that the backward error is bounded by the forward error and thus it is enough to satisfy \eqref{eq:fe_converg_bound1} or \eqref{eq:fe_converg_bound2}. 

We can now comment more on the conditions \eqref{eq:fe_converg_bound1} and \eqref{eq:fe_converg_bound2}. We assume that $\Omega$ is chosen to be a subspace embedding. We make the following observations:
\begin{itemize}
    \item The term $\psi(\beta,\Omega A)$ is expected to dominate in all the conditions.
    \item If the conditions in Theorem~\ref{th:finite_recision_error_R} are satisfied, that is, $\kappa(A) < u_s^{-1}$ and $\kappa(A) < u_{QR}^{-1}$, then $\beta$ grows moderately with the problem dimension, the values $\Vert \Omega A \Vert_2$, $\Vert (\Omega A)^{\dagger} \Vert_2$, and $\kappa_2(\Omega A) $ stay close to $\Vert  A \Vert_2$, $\Vert A^{\dagger} \Vert_2$, and $\kappa_2(A)$, respectively, and both $\Vert ( AR^{-1})^{\dagger}\Vert_2$ and $\Vert AR^{-1}\Vert_2$ are close to 1, and thus we can expect LSIR to converge when $\kappa(A)$ is safely less than $u^{-1}$. This is in contrast to the convergence theory for iterative refinement for linear systems of equations where we require $\kappa(\Atld)$ to be safely less than $u^{-1}$.
    \item Note that setting $u_s > u$ is thus allowed by the theory, but we can guarantee LSIR convergence only when the low precision does not regularize the problem too much, that is, we need to choose $u_s$ according to $\kappa_2(A)$.
    \item Assume that $\kappa(A) > u_s^{-1}$. Then we can replace $A$ with its lower precision version $A_s$ in the results of Theorem~\ref{th:finite_recision_error_R} as discussed in Section~\ref{sec:sketching_regularized}. If the optimal scaling for the preconditioned augmented system is used, then $\kappa_2(M_L^{-1} \Atld M_R^{-1}) \leq 2 \kappa_2(A R_s^{-1})$ holds and $\kappa_2(M_L)$ depends on the conditioning of $R_s$.
    The regularization by casting to a lower precision can give $\kappa_2(R_s) \leq \kappa_2(R)$, however $\kappa_2(A R_s^{-1})$ can be significantly larger than $\kappa_2(A R^{-1})$. Some information useful for preconditioning can be obtained from the regularized $A_s$, but caution should be exercised when considering $\kappa_2(A) \gg u_s^{-1}$.
\end{itemize}

\section{Numerics for dense problems}\label{sec:numerics}

We illustrate the analysis with simple numerical experiments performed in MATLAB R2023b\footnote{The code is available at \url{https://github.com/dauzickaite/LSIRrndprec/}}. The aim of the experiments is to show that, as predicted by the theory, the mixed precision sketched preconditioner reduces the condition number of coefficient matrices and enables LSIR to converge. Providing detailed recommendations for efficient implementation is out of scope of this paper.

The least-squares problem \eqref{eq:LS_problem} is constructed with a synthetic dense $A$ generated as a \textit{'randsvd'} matrix from the MATLAB test matrices gallery with $m=10^3$ and $n = 10^2$, geometrically distributed singular values, and various choices of $\kappa_2(A)$. The right-hand side $b$ is a random vector with entries drawn from a uniform distribution in the interval $(0,1)$ and normalized to have a unit norm. Such a right-hand side gives $\Vert r \Vert_2 = \Vert b - Ax \Vert_2 \approx 1$ and thus the sensitivity of the least-squares problem depends on $\kappa_2(A)^2$; we note that when $\Vert r \Vert_2$ is small the sensitivity depends on $\kappa_2(A)$ instead; see, e.g., \cite[Section 20.1]{high:ASNA2}. 
We only test this large residual setting since it is the case where the augmented system approach to LSIR is expected to be most advantageous over other LSIR approaches in terms of refining the solution $x$; see, e.g., \cite{carson2024comparison}.
We generate the sketching matrix as $\Omega = (4n)^{-1/2} G$, where $G$ is a random $4n \times m$ matrix with entries drawn from a standard normal distribution. There is no scaling for the augmented system, that is, $\alpha=1$.

The precisions in Algorithms~\ref{alg:generating_R} and \ref{alg:augmented_approach} are set so that $u_s \leq u$, $u_{QR} = u$, $u_r = u^2$. FGMRES is run with $u_A = u_L = u_R = u$. If LSIR does not converge in 30 iterations and $\kappa_2(A) < u_s^{-1}$, then we set $u_A = u_L = u_R = u^2$ and rerun the refinement loop, that is, steps~\ref{irstep:while_start}-\ref{irstep:while_end} of Algorithm~\ref{alg:augmented_approach}. 
We test the following settings: $(u_s, u, u_r)=$ (half, single, double), (single, single, double), (half, double, quad), (single, double, quad), (double, double, quad); note that the unit roundoff is $2^{-11}$ for half, $2^{-24}$ for single, $2^{-53}$ for double, and $2^{-113}$ for quad. MATLAB native single and double precisions are used, half precision is simulated via the \textit{chop} library \cite{HighamChop}, and quadruple (quad) precision is simulated via the Advanpix Multiprecision Computing Toolbox \cite{advanpix}. When $u$ is set to single, we store $A$ and $b$ in single precision.

In Algorithm~\ref{alg:augmented_approach}, the MATLAB implementation of LSQR is run for $2n$ iterations or until the tolerance reaches $10^{-6}$ if $u$ is set to single and $10^{-12}$ if $u$ is set to double. We compute the `true' solution $x^*$ to \eqref{eq:LS_problem} using MATLAB backslash in arithmetic that is simulated to be accurate to 64 digits using the Advanpix toolbox. The same accuracy is used to compute the `true' residual $r^* = b - A x^*$. LSIR is run for 30 iterations or until the relative errors in both $x$ and $r$ satisfy
\begin{equation*}
    \frac{\Vert r^* - \rhat_i \Vert_2}{\Vert r^* \Vert_2} \leq 4 u \quad \textrm{and} \quad \frac{\Vert x^* - \xhat_i \Vert_2}{\Vert x^* \Vert_2} \leq 4u.
\end{equation*}
FGMRES is terminated after 50 iterations or when the tolerance reaches the same values as for LSQR.

We compute the condition numbers and norms of the preconditioned matrices. The results when QR in Algorithm~\ref{alg:generating_R} is computed in single precision (Figure~\ref{fig:condition_numbers_uqr_single})
and in double precision (Figure~\ref{fig:condition_numbers_uqr_double})
show that as long as $\kappa_2(A) < u_s^{-1}$, preconditioning with $\Rhat$ keeps the condition numbers of both $M_L^{-1} \Atld M_R^{-1}$ and $ A \Rhat^{-1}$ at $\mathcal{O}(1)$. If, however, $\kappa_2(A) > u_s^{-1}$, then $\kappa_2(M_L^{-1} \Atld M_R^{-1})$ grows significantly due to the size of $\Vert (A \Rhat^{-1})^{\dagger} \Vert_2$; the value of $\Vert A\Rhat^{-1} \Vert_2$ always stays close to 2. These results agree with our theoretical observations. In the case when $\kappa_2(A) > u_s^{-1}$, $\Rhat$ is still effective in reducing $\kappa_2(A\Rhat^{-1})$ compared to $\kappa_2(A)$ by approximately a factor of $u_s^{-1}$. Note that as predicted by our analysis, we obtain $\Rhat$ of the same quality when we use both $u_{QR} = u_s$ and $u_{QR} < u_s$.

The initial solves with LSQR give solutions and residuals of similar quality in the same or very similar number of iterations if $\kappa_2(A)$ is sufficiently smaller than $u_s^{-1}$; see Tables~\ref{tab:lsqr_solve_us_half_uqr_single} and \ref{tab:lsqr_solve_u_double}. If, however, this is not the case, we clearly obtain a worse preconditioner and LSQR needs significantly more iterations to converge or convergence is not reached in the preset number of iterations. Note that although an impractical number of iterations are required for $u_s$ set to a precision such that $\kappa_2(A) > u_s^{-1}$, the method is still able to reach the same accuracy as with $u_s$ set to a higher precision. 

LSIR convergence results with $u$ set to single and double are presented in Tables~\ref{tab:lsir_iterations_u_single} and \ref{tab:lsir_iterations_u_double}, respectively. Note that the iterative refinement process converges in the cases where our theoretical analysis holds, that is, when $\kappa_2(A) < u_s^{-1}$. In order to achieve convergence when $\kappa_2(A)$ is close to $u^{-1}$, we have to increase the precisions in FGMRES for computing the matrix-vector products and applying the preconditioner, 
and possibly allow more FGMRES iterations in every LSIR iteration; this is needed for FGMRES to reach the required backward error, see, e.g., \cite{amestoy2024five} and \cite{carson2024stability}. We note that it is also possible to achieve LSIR convergence when $\kappa_2(A) > u_s^{-1}$ by significantly increasing the maximum number of FGMRES iterations, setting lower tolerance for FGMRES and/or increasing $u_A$, $u_L$ and $u_R$ in FGMRES. These combinations of $u_s$ and $\kappa_2(A)$ are however not covered by our analysis and the FGMRES parameters are highly problem dependent and require tuning in practice.

\begin{figure}[htbp]
\setlength\figureheight{6cm}
\setlength\figurewidth{0.99\textwidth}
    \centering
    \begin{subfigure}[h]{0.7\linewidth}
    \centering
\begin{tikzpicture}
\begin{axis}[
width = 0.5\figurewidth,
height = \figureheight,
at = {(0\figurewidth,0.5\figureheight)},
    xmode=log,
    ymode=log,
    xlabel={$\kappa_2(A)$},
    legend style={
    fill=none,
    font=\small
    },
     legend pos=outer north east,
    title={}
]
\addplot table[row sep=\\] {
1e+00  7.44e+00 \\
1e+01  7.44e+00 \\
1e+02  7.45e+00 \\
1e+03  9.29e+00 \\
1e+04  3.40e+02 \\
1e+05  2.48e+04 \\
1e+06  1.80e+06 \\
1e+07  1.57e+08 \\
};
\addlegendentry{$\kappa_2(M_L^{-1} \Atld M_R^{-1})$}

\addplot table[row sep=\\] {
1e+00  2.03e+00 \\
1e+01  2.03e+00 \\
1e+02  2.03e+00 \\
1e+03  1.86e+00 \\
1e+04  1.79e+00 \\
1e+05  1.67e+00 \\
1e+06  1.67e+00 \\
1e+07  1.60e+00 \\
};
\addlegendentry{$\Vert A \Rhat^{-1} \Vert_2$}

\addplot table[row sep=\\] {
1e+00  1.46e+00 \\
1e+01  1.46e+00 \\
1e+02  1.46e+00 \\
1e+03  1.74e+00 \\
1e+04  1.19e+01 \\
1e+05  1.05e+02 \\
1e+06  8.97e+02 \\
1e+07  8.48e+03 \\
};
\addlegendentry{$\Vert (A \Rhat^{-1})^{\dagger} \Vert_2$}

\addplot table[row sep=\\] {
1e+00  2.97e+00 \\
1e+01  2.96e+00 \\
1e+02  2.96e+00 \\
1e+03  3.24e+00 \\
1e+04  2.15e+01 \\
1e+05  1.76e+02 \\
1e+06  1.50e+03 \\
1e+07  1.36e+04 \\
};
\addlegendentry{$\kappa_2( A \Rhat^{-1})$}

\addplot[color=green, mark=diamond] table[row sep=\\] {
1e+00  2.62e+00 \\
1e+01  1.63e+02 \\
1e+02  1.62e+04 \\
1e+03  1.69e+06 \\
1e+04  1.34e+08 \\
1e+05  4.49e+08 \\
1e+06  9.59e+08 \\
1e+07  3.10e+08 \\
};
\addlegendentry{$\kappa_2(\Atld)$}

\end{axis}
\end{tikzpicture}
\caption{$u_s$ half, $u_{QR}$ single}
\end{subfigure}\\
\begin{subfigure}[h]{0.5\textwidth}
\centering
\begin{tikzpicture}
\begin{axis}[
width = 0.5\figurewidth,
height = \figureheight,
at = {(0\figurewidth,0.5\figureheight)},
    xmode=log,
    xlabel={$\kappa_2(A)$},
    title={}
]

\addplot table[row sep=\\, x index=0, y index=1] {
1e+00 7.45e+00 \\
1e+01 7.45e+00 \\
1e+02 7.45e+00 \\
1e+03 7.45e+00 \\
1e+04 7.45e+00 \\
1e+05 7.45e+00 \\
1e+06 7.44e+00 \\
1e+07 7.65e+00 \\
};

\addplot table[row sep=\\, x index=0, y index=1] {
1e+00 2.03e+00 \\
1e+01 2.03e+00 \\
1e+02 2.03e+00 \\
1e+03 2.03e+00 \\
1e+04 2.03e+00 \\
1e+05 2.03e+00 \\
1e+06 2.03e+00 \\
1e+07 2.03e+00 \\
};

\addplot table[row sep=\\, x index=0, y index=1] {
1e+00 1.46e+00 \\
1e+01 1.46e+00 \\
1e+02 1.46e+00 \\
1e+03 1.46e+00 \\
1e+04 1.46e+00 \\
1e+05 1.46e+00 \\
1e+06 1.46e+00 \\
1e+07 1.49e+00 \\
};

\addplot table[row sep=\\, x index=0, y index=1] {
1e+00 2.97e+00 \\
1e+01 2.97e+00 \\
1e+02 2.97e+00 \\
1e+03 2.97e+00 \\
1e+04 2.97e+00 \\
1e+05 2.97e+00 \\
1e+06 2.97e+00 \\
1e+07 3.01e+00 \\
};
\end{axis}
\end{tikzpicture}
\caption{$u_s$ single, $u_{QR}$ single}
\end{subfigure}
\caption{Condition numbers and norms when precisions $u_s$ and $u_{QR}$ in Algorithm~\ref{alg:generating_R} are set to half/single and single, respectively. $\kappa_2(\Atld)$ is the same in both cases. Note the different scales for the y axis.}
    \label{fig:condition_numbers_uqr_single}
\end{figure}
\begin{figure}[htbp]
\setlength\figureheight{6cm}
\setlength\figurewidth{0.99\textwidth}
    \centering
    \begin{subfigure}[b]{0.8\textwidth}
        \centering
        \begin{tikzpicture}
        \begin{axis}[
width = 0.5\figurewidth,
height = \figureheight,
at = {(0\figurewidth,0.5\figureheight)},
            xmode=log,
            ymode=log,
            ytick = {1e0,1e3,1e6,1e9,1e12,1e15,1e18},
            xlabel={$\kappa_2(A)$},
            ylabel={},
            legend style={
                fill=none,
                font=\small
            },
            legend pos=outer north east,
        ]
        \addplot table[row sep=\\] {
            x y \\
            1e+02  7.45e+0 \\
            1e+04  3.40e+02 \\
            1e+06  1.81e+06 \\
            1e+08  1.58e+10 \\
            1e+10 1.31e+14 \\
            1e+12 4.81e+17 \\
            1e+14 1.03e+17 \\
            1e+15 1.30e+17 \\
        };
        \addlegendentry{$\kappa_2(M_L^{-1} \Atld M_R^{-1})$}

        \addplot table[row sep=\\] {
            x y \\
            1e+02  2.02 \\
            1e+04  1.80 \\
            1e+06  1.67 \\
            1e+08  1.65 \\
            1e+10 1.55 \\
            1e+12 1.56 \\
            1e+14 1.47 \\
            1e+15 1.49 \\
        };
        \addlegendentry{$\Vert A \Rhat^{-1} \Vert_2$}

        \addplot table[row sep=\\] {
            x y \\
            1e+02  1.46 \\
            1e+04  1.19e+01 \\
            1e+06  8.98e+02 \\
            1e+08  8.43e+04 \\
            1e+10 7.82e+06 \\
            1e+12 6.98e+08 \\
            1e+14 6.32e+10 \\
            1e+15 6.72e+11 \\
        };
        \addlegendentry{$\Vert (A \Rhat^{-1})^{\dagger} \Vert_2$}

        \addplot table[row sep=\\] {
            x y \\
            1e+02  2.96 \\
            1e+04  2.15e+01 \\
            1e+06  1.50e+03 \\
            1e+08  1.39e+05 \\
            1e+10 1.22e+07 \\
            1e+12 1.09e+09 \\
            1e+14 9.30e+10 \\
            1e+15 1.00e+12 \\
        };
        \addlegendentry{$\kappa_2( A \Rhat^{-1})$}

         \addplot[color=green, mark=diamond] table[row sep=\\] {
            x y \\
            1e+02  1.62e+04 \\
            1e+04  1.62e+08 \\
            1e+06  1.62e+12 \\
            1e+08  1.39e+16 \\
            1e+10 7.00e+17 \\
            1e+12 1.11e+17 \\
            1e+14 1.44e+17 \\
            1e+15 1.75e+17 \\
        };
        \addlegendentry{$\kappa_2(\Atld)$}
        \end{axis}
        \end{tikzpicture}
        \caption{$u_s$ half, $u_{QR}$ double}
    \end{subfigure}
    \begin{subfigure}[b]{0.5\textwidth}
        \centering
        \begin{tikzpicture}
        \begin{axis}[
width = 0.48\figurewidth,
height = \figureheight,
at = {(0\figurewidth,0.5\figureheight)},
            xmode=log,
            ymode=log,
            ytick = {1e0,1e3,1e6,1e9,1e12,1e15,1e18},
            xlabel={$\kappa_2(A)$},
        ]
        \addplot table[row sep=\\] {
            x y \\
            1e2  7.45 \\
            1e4  7.45 \\
            1e6  7.44 \\
            1e8  5.64e1 \\
            1e10 3.98e5 \\
            1e12 3.37e9 \\
            1e14 2.71e13 \\
            1e15 2.71e15 \\
        };

        \addplot table[row sep=\\] {
            x y \\
            1e2  2.03 \\
            1e4  2.03 \\
            1e6  2.03 \\
            1e8  1.90 \\
            1e10 1.81 \\
            1e12 1.80 \\
            1e14 1.69 \\
            1e15 1.67 \\
        };

        \addplot table[row sep=\\] {
            x y \\
            1e2  1.46 \\
            1e4  1.46 \\
            1e6  1.46 \\
            1e8  4.67 \\
            1e10 4.09e2 \\
            1e12 3.77e4 \\
            1e14 3.46e6 \\
            1e15 3.51e7 \\
        };

        \addplot table[row sep=\\] {
            x y \\
            1e2  2.97 \\
            1e4  2.97 \\
            1e6  2.97 \\
            1e8  8.92 \\
            1e10 7.41e2 \\
            1e12 6.80e4 \\
            1e14 5.86e6 \\
            1e15 5.86e7 \\
        };

        \addplot[color=green, mark=diamond] table[row sep=\\] {
            x y \\
            1e+02  1.62e+04 \\
            1e+04  1.62e+08 \\
            1e+06  1.62e+12 \\
            1e+08  1.39e+16 \\
            1e+10 7.00e+17 \\
            1e+12 1.11e+17 \\
            1e+14 1.44e+17 \\
            1e+15 1.75e+17 \\
        };
        \end{axis}
        \end{tikzpicture}
        \caption{$u_s$ single, $u_{QR}$ double}
    \end{subfigure}\hfill
    \begin{subfigure}[b]{0.45\textwidth}
        \centering
        \begin{tikzpicture}
        \begin{axis}[
width = 0.47\figurewidth,
height = \figureheight,
at = {(0\figurewidth,0.5\figureheight)},
            xmode=log,
            xlabel={$\kappa_2(A)$}
        ]
        \addplot table[row sep=\\] {
            x y \\
            1e2  7.45 \\
            1e4  7.45 \\
            1e6  7.45 \\
            1e8  7.45 \\
            1e10 7.45 \\
            1e12 7.45 \\
            1e14 7.45 \\
            1e15 7.43 \\
        };

        \addplot table[row sep=\\] {
            x y \\
            1e2  2.03 \\
            1e4  2.03 \\
            1e6  2.03 \\
            1e8  2.03 \\
            1e10 2.03 \\
            1e12 2.03 \\
            1e14 2.03 \\
            1e15 2.03 \\
        };

        \addplot table[row sep=\\] {
            x y \\
            1e2  1.46 \\
            1e4  1.46 \\
            1e6  1.46 \\
            1e8  1.46 \\
            1e10 1.46 \\
            1e12 1.46 \\
            1e14 1.46 \\
            1e15 1.46 \\
        };

        \addplot table[row sep=\\] {
            x y \\
            1e2  2.97 \\
            1e4  2.97 \\
            1e6  2.97 \\
            1e8  2.97 \\
            1e10 2.97 \\
            1e12 2.97 \\
            1e14 2.97 \\
            1e15 2.96 \\
        };
        \end{axis}
        \end{tikzpicture}
        \caption{$u_s$ double, $u_{QR}$ double}
    \end{subfigure}

    \caption{As in Figure~\ref{fig:condition_numbers_uqr_single}, but $u_{QR}$ is set to double, and $u_s$ is set to half/single/double. $\kappa_2(\Atld)$ is the same in all three cases. Note the different scales for the y axis.}
    \label{fig:condition_numbers_uqr_double}
\end{figure}

\begin{table}[]
    \centering
    \begin{tabular}{c|c|c|c|c|c|c}
       $\kappa_2(A)$  & \multicolumn{2}{c}{LSQR it.} & \multicolumn{2}{|c|}{$\Vert x^* - \xhat \Vert_2 /\Vert x^* \Vert_2$} & \multicolumn{2}{c}{$\Vert r^* - \rhat \Vert_2 /\Vert r^* \Vert_2$ }\\
         & half &  single  & half &   single  &  half &  single \\
          \hline
1e+00 & 16  & 16 & 9.42e-06 & 9.46e-06 & 3.11e-06 & 3.13e-06 \\
1e+01 & 16  & 16 & 1.04e-05 & 9.99e-06 & 3.20e-06 & 3.13e-06 \\
1e+02 & 16  & 16 & 1.24e-05 & 1.04e-05 & 3.24e-06 & 3.15e-06 \\
1e+03 & 17  & 16 & 7.37e-05 & 3.40e-05 & 8.62e-06 & 6.81e-06 \\
1e+04 & 85  & 16 & 6.69e-04 & 4.24e-04 & 7.12e-05 & 4.83e-05 \\
1e+05 & 200 & 16 & 5.20e-02 & 4.72e-03 & 5.22e-03 & 4.40e-04 \\
1e+06 & 200 & 16 & 8.66e-01 & 3.52e-02 & 7.48e-02 & 3.75e-03 \\
1e+07 & 200 & 17 & 9.96e-01 & 2.93e-01 & 1.37e-01 & 3.11e-02
    \end{tabular}
    \caption{LSQR iteration counts and relative errors in the solution and residual when $u_s$ is set to half and single; $u_{QR}$ and $u$ are set to single.}
    \label{tab:lsqr_solve_us_half_uqr_single}
\end{table}

\begin{table}[]
    \centering
 \resizebox{\textwidth}{!}{   \begin{tabular}{c|c|c|c|c|c|c|c|c|c}
      $\kappa_2(A)$  & \multicolumn{3}{c}{LSQR it.} & \multicolumn{3}{|c|}{$\Vert x^* - \xhat \Vert_2 /\Vert x^* \Vert_2$} & \multicolumn{3}{c}{$\Vert r^* - \rhat \Vert_2 /\Vert r^* \Vert_2$ }\\
      & half & single & double  & half & single & double  & half & single & double\\
          \hline
1e+02 & 31  & 31  & 31 & 2.41e-11 & 2.40e-11 & 2.40e-11 & 7.35e-12 & 6.57e-12 & 6.57e-12 \\
1e+04 & 110 & 32  & 32 & 2.19e-12 & 1.31e-11 & 1.28e-11 & 3.79e-12 & 4.46e-12 & 4.50e-12 \\
1e+06 & 200 & 32  & 32 & 7.45e-01 & 5.91e-11 & 5.14e-11 & 5.21e-02 & 9.02e-12 & 8.38e-12 \\
1e+08 & 200 & 54  & 32 & 1e00 & 6.01e-09 & 3.43e-09 & 1.44e-01 & 8.40e-10 & 4.89e-10 \\
1e+10 & 200 & 200 & 32 & 1e00     & 1.59e-01 & 4.10e-07 & 1.75e-01 & 1.25e-02 & 3.76e-08 \\
1e+12 & 200 & 200 & 33 & 1e00     & 9.99e-01 & 3.61e-05 & 2.21e-01 & 1.10e-01 & 4.27e-06 \\
1e+14 & 200 & 200 & 34 & 1e00     & 1e00     & 2.94e-03 & 2.34e-01 & 1.44e-01 & 3.58e-04 \\
1e+15 & 200 & 200 & 34 & 1e00     & 1e00     & 8.35e-02 & 2.36e-01 & 1.53e-01 & 5.17e-03
    \end{tabular}}
    \caption{As in Table~\ref{tab:lsqr_solve_us_half_uqr_single}, but both $u_{QR}$ and $u$ are set to double, and $u_s$ is set to half, single, and double. }
    \label{tab:lsqr_solve_u_double}
\end{table}

\begin{table}[]
    \centering
    \begin{tabular}{c|c|c|c|c}
     $\kappa_2(A)$     & \multicolumn{2}{c}{ LSIR it.} & \multicolumn{2}{|c}{FGMRES it.}\\
         &  half &  single  &  half &  single  \\
         \hline
1e+00 & 1  & 1 & 39  & 39                    \\
1e+01 & 1  & 1 & 39  & 39                    \\
1e+02 & 1  & 1 & 39  & 39                    \\
1e+03 & 1  & 1 & 45  & 40                    \\
1e+04 & 9 & 1 & 450 & 41                    \\
1e+05 & -  & 2 & -   & 78                    \\
1e+06 & -  & 2* & -   & 77* \\
1e+07 & -  & 2* & -   & 78*         
    \end{tabular}
    \caption{LSIR iterations and the total count of FGMRES iterations within LSIR (Algorithm~\ref{alg:augmented_approach}) for different values of $u_s$. Here $u$ is set to single, and $u_s$ is either half or single. - denotes that LSIR did not converge in 30 iterations. * denotes when $u_A$, $u_L$ and $u_R$ in FGMRES are set to double.}
    \label{tab:lsir_iterations_u_single}
\end{table}

\begin{table}[]
    \centering
    \begin{tabular}{c|c|c|c|c|c|c}
     $\kappa_2(A)$     & \multicolumn{3}{c}{ LSIR it.} & \multicolumn{3}{|c}{FGMRES it.}\\
         &  half &  single  & double &  half &  single  & double \\
         \hline
    1e+02 & 1 & 1 & 1  & 50  & 50 & 50  \\
1e+04 & 7 & 1 & 1  & 350 & 50 & 50  \\
1e+06 & - & 1 & 1  & -   & 50 & 50  \\
1e+08 & - & - & 1  & -   & -  & 50  \\
1e+10 & - & - & 2  & -   & -  & 100 \\
1e+12 & - & - & 2*  & -   & -  & 138* \\
1e+14 & - & - & 2*  & -   & -  & 137* \\
1e+15 & - & - & 6* & -   & -  & 409*
    \end{tabular}
    \caption{As in Table~\ref{tab:lsir_iterations_u_single}, but $u$ is set to double, and $u_s$ is either half, single, or double. - denotes that LSIR did not converge in 30 iterations. * denotes when $u_A$, $u_L$ and $u_R$ in FGMRES are set to quad and the maximum number of FGMRES iterations is increased to 80.}
    \label{tab:lsir_iterations_u_double}
\end{table}

\section{Conclusions}\label{sec:conclusions}

In this paper, we provide theoretical analysis of a mixed precision approach to generating a sketched preconditioner for least-squares problems. We show that the computed $R$-factor $\Rhat$ of the sketched problem is close to the exact $R$-factor when the sketching precision $u_s$ is chosen such that $u_s \kappa_2(A)<1$ is satisfied, that is, $u_s$ is a precision which does not regularize the smallest singular values of $A$. Then $\Rhat$ is an effective preconditioner for iterative least-squares solvers. If we set $u_s$ such that $u_s \kappa_2(A)>1$ and thus the regularization because of sketching in lower precision is significant, then $\Rhat$ may still be effective in reducing $\kappa_2(A\Rhat^{-1})$ compared to $\kappa_2(A)$, however it does not appear to be effective in ensuring the convergence of an iterative least-squares solver in a small number of iterations. In such a setting, the practitioner should thus carefully evaluate if for their particular application the savings because of sketching in lower precision are enough to offset the cost of additional solver iterations.

If the computed solution and the residual are required to be of high quality and thus an iterative refinement approach is necessary, our theoretical analysis shows that if we set $u_s$ such that $u_s \kappa_2(A)<1$, the computed preconditioner can be used to ensure the convergence of an FGMRES-based LSIR scheme without the need to scale the augmented system. Note that the sketching precision can be lower than the working precision $u$ or equal to it, so if $u_s \kappa_2(A)<1$, Algorithm \ref{alg:augmented_approach} is guaranteed to converge to its limiting accuracy, which depends on $u$ and $u_r$. If $u_s \kappa_2(A)<1$ is not satisfied, then FGMRES-based LSIR can still converge as observed in numerical experiments, but the cost of iterative refinement grows significantly and no theoretical guarantees are provided.

Previous work used a full QR factorization computed in some low precision $u_f$ to precondition a GMRES-based LSIR scheme, and in this case convergence can be guaranteed even when $u_f \kappa_2(A)>1$ \cite{carson2020three}. This approach, however, requires an expensive-to-compute optimal scaling for the augmented system. We note that computing the full QR factorization even in low precision is expensive and may not be feasible in some applications. Our work thus shows that one can instead use modern alternatives, such as randomized QR factorizations, to construct preconditioners for GMRES-based iterative refinement for least-squares problems that have significantly more rows than columns.

\bibliographystyle{siam}
\bibliography{bibl}

\end{document}